\documentclass[12pt,reqno]{article}

\usepackage[usenames]{color}
\usepackage{subcaption}
\usepackage{amssymb}
\usepackage{amsmath}
\usepackage{amsthm}
\usepackage{amsfonts}
\usepackage{amscd}
\usepackage{graphicx}
\usepackage{tikz}
\usetikzlibrary{automata, positioning}
\usepackage{enumitem}
\usepackage{mathtools}
\usepackage{geometry}
\usepackage{bm}

\usepackage[colorlinks=true,
linkcolor=webgreen,
filecolor=webbrown,
citecolor=webgreen]{hyperref}

\definecolor{webgreen}{rgb}{0,.5,0}
\definecolor{webbrown}{rgb}{.6,0,0}

\usepackage{color}
\usepackage{fullpage}
\usepackage{float}
\usepackage{hyperref}

\usepackage{graphics}
\usepackage{latexsym}
\usepackage{epsf}
\usepackage{breakurl}

\usepackage{hyperref,xcolor}
\usepackage[T1]{fontenc}
\usepackage{pdfsync}
\usepackage{yfonts}      
\usepackage{tcolorbox}   
\usepackage{lipsum}

\newcommand{\seqnum}[1]{\href{https://oeis.org/#1}{\rm \underline{#1}}}

\begin{document}

%\begin{center}
%\epsfxsize=4in
%\leavevmode\epsffile{logo129.eps}
%\end{center}

\theoremstyle{plain}
\newtheorem{theorem}{Theorem}
\newtheorem{corollary}[theorem]{Corollary}
\newtheorem{lemma}[theorem]{Lemma}
\newtheorem{proposition}[theorem]{Proposition}
\newtheorem{remarks}[theorem]{Remarks}

\theoremstyle{definition}
\newtheorem{definition}[theorem]{Definition}
\newtheorem{example}[theorem]{Example}
\newtheorem{conjecture}[theorem]{Conjecture}

\theoremstyle{remark}
\newtheorem{remark}[theorem]{Remark}

% for edits
%Gandhar
        \newcommand{\Gandhar}[1]{\marginpar{\small \raggedright \textcolor{orange}{{\bfseries Gandhar:} #1}}}
        \newcommand{\gandhar}[1]{\textcolor{orange}{ #1}}
\begin{center}
\vskip 1cm{\LARGE\bf On Cloitre's hiccup sequences
}
\vskip 1cm
\large
Robbert Fokkink\\
Technical University Delft\\ Faculty of Mathematics\\ Mekelweg 4, 2628CD\\
The Netherlands\\
\href{mailto:email}{\tt r.j.fokkink@tudelft.nl} \\
{\small
\href{https://orcid.org/0000-0001-8347-4110}{ORCID: 0000-0001-8347-4110}}\\
\vskip 0.2cm
Gandhar Joshi\\
School of Mathematics and Statistics\\ 
The Open University\\
Walton Hall\\
Milton Keynes, MK7 6AA\\ 
UK\\
\href{mailto:gandhar.joshi@open.ac.uk}{\tt gandhar.joshi@open.ac.uk}\\
{\small
\href{https://orcid.org/0000-0002-6960-182X}{ORCID: 0000-0002-6960-182X}}
\end{center}

\vskip .2 in
\begin{abstract}
In 2003, Benoit Cloitre entered a family of sequences in the OEIS that we call \emph{hiccup} sequences.
We collect the various claims, observations, and proofs of properties of these
sequences that have been entered in the OEIS over the years, and present a unified approach,
inspired by a remarkable theorem of Bosma, Dekking, and Steiner.
\end{abstract}

%\maketitle

Aronson's sequence 
\[
1, 4, 11, 16, 24, 29, 33, 35, 39, 45, 47, 51, 56, 58, 62, 64, \ldots,
\]
is entry \seqnum{A005224} in the On-Line Encyclopedia of Integer Sequences.
It comes from the self-referential sentence
\begin{center}
\emph{T is The firsT, fourTh, elevenTh, sixTeenTh, ... letter in this sentence,}    
\end{center}
which keeps track of the positions of the $T$'s.
Spaces and punctuation are ignored. In the English language, ordinal numbers, with the notable exception of second,
contain enough $T$'s to keep the sentence going on forever. 
Douglas Hofstadter~\cite{hofstadter2008}
attributed this sentence to the British clinical pharmacologist Jeffrey Aronson. 

The words in Aronson's sentence generate numbers
that generate words. It is a Munchausen manner to produce an integer sequence, and
there are many different ways in which one can vary on this theme.
Cloitre, Sloane, and Vandermast~\cite{cloitre2003} considered several methods to generate self-referential sequences.
We are interested in one specific type.
Twenty years ago, Benoit Cloitre entered several sequences of this type
in the OEIS. One of these is \seqnum{A086398}, and its first entries are
\begin{equation}\label{eq:1}
1, 5, 7, 9, 11, 15, 17, 21, 23, 27, 29, 33, 35, 37, 39, 43, 45, 49, 51, 53,\ldots .
\end{equation}
where $a(n)=a(n-1)+4$ if $n-1$ is already in the sequence, and otherwise $a(n)=a(n-1)+2$,
starting from $a(1)=1$. Actually, the definition in the OEIS is slightly different from ours,
but it is not hard to check that it is equivalent. The differences 
$a(n)-a(n-1)$ take two values that appear in some sort of rhythm, just like
bouts of hiccups~\cite[p96]{hiccup}, and that is why we call them \emph{hiccup sequences}.
We have collected many hiccup sequences from the OEIS in Table~\ref{tbl:1} below.

If we vary the definition of \seqnum{A086398} a bit 
and put $a(n)=a(n-1)+4$ if $n$ is in the sequence and $a(n)=a(n-1)+2$
otherwise,
then we get \seqnum{A080903}, which has a lot of overlap with \seqnum{A086398}
\begin{equation}\label{eq:2}
1, 3, 7, 9, 11, 13, 17, 19, 23, 25, 29, 31, 35, 37, 39, 41, 45, 47, 51, 53,\ldots .
\end{equation}
This overlap of the two sequences will be cleared up in Section~3.
Another hiccup sequence in the OEIS is~\seqnum{A064437}. Robert Israel empirically found a closed formula for this sequence:
\begin{equation}\label{eq:3}
a(n)=\left\lceil\left(1+\sqrt{2}\right)(n-1)+\frac{1}{2+\sqrt{2}}\right\rceil    
\end{equation}
and checked its validity for the first hundred entries. 
The reader may recognize that this is a non-homogeneous Beatty sequence~\cite{beatty} with slope
$1+\sqrt 2$ and offset $\frac{1}{2+\sqrt{2}}$. Similar formulas can be found on the
OEIS for other hiccup sequences.

It is possible to define hiccup sequences that are not increasing, or have first differences
that take more than two values. There are plenty of such hiccup sequences in the OEIS as well.
For example, the hiccup sequence \seqnum{A080782} produces a permutation of the natural numbers.
Other examples are \seqnum{A079354}, \seqnum{A080458}, \seqnum{A080900},
and \seqnum{A111202}.
A hiccup sequence with three gaps is \seqnum{A080574}.
These are all very interesting sequences, but we
do not consider them in this paper. 
We only consider hiccup sequences that are strictly increasing.
In a recent paper, Cloitre~\cite{cloitre2025} has revisited his sequences
and also considers non-decreasing hiccup sequences.

Our paper is organized as follows. In section 1, we 
define the $(k,x,y,z)$-hiccup sequences and
give an overview of all the results and
claims that were entered into the OEIS.
We prove that hiccup sequences are morphic in section 2. 
We use this result in section 3 to retrieve
Cloitre's recent theorem that hiccup sequences are Beatty sequences if $|y-z|=1$ 
under certain further restrictions on the parameters. 
It turns out that $(0,x,1,z)$-sequences
are examples of morphic non-automatic sequences, which feature in a recent overview article
by 
Allouche, Shallit, and Yassawi~\cite{allouche2022}. 
We move beyond Beatty sequences and settle a conjecture
of Kimberling, combining the automatic theorem prover \texttt{Walnut}~\cite{mousavi}
with results of Ollinger and others on Dumont-Thomas numeration~\cite{carton2025,licofage}.
In section 4 we generalize a conjecture of Bosma, Dekking, and Steiner~\cite{bosma2018} on
a connection between continued fractions and Beatty sequences.
We settle a specific instance of that generalized conjecture.

\section{Definition, literature, and OEIS insights}

We call an increasing sequence a \((j, x, y, z)\)-\emph{hiccup sequence} if it starts at \(x\), and from then on, you add \(y\) if the 
number \(n - j\) already showed up in the sequence, and \(z\) if it didn’t.

\begin{definition}[\((j, x, y, z)\)-hiccup sequence]
Let \( a(n) \) be a strictly increasing sequence of integers, and let \( j, x \in \mathbb{Z}_{\geq 0} \),
and \( x, y \in \mathbb{Z}_{\geq 1} \) with $y\not=z$.  
We say that \( a(n) \) is the \emph{\((j, x, y, z)\)-hiccup sequence} if it satisfies:
\begin{align*}
a(1) &= x, \\
a(n) &= 
\begin{cases}
a(n - 1) + y, & \text{if } n - j \in \{a(k)\}_{k < n}, \\[4pt]
a(n - 1) + z, & \text{otherwise},
\end{cases}
\quad \text{for } n \geq 2.
\end{align*}
\end{definition} 
In a recent paper, Cloitre~\cite{cloitre2025} considers $(0,x,y,z)$-hiccup sequences
in which $y$ or $z$ can be zero. Our results partially overlap Cloitre's, 
but are obtained by a different approach.

\begin{table}[h!]
\centering
\small
\begin{tabular}{|c|c|c|c|c||c|c|c|c|c|}
\hline
\textbf{OEIS} & $j$ & $x$ & $y$ & $z$ &
\textbf{OEIS} & $j$ & $x$ & $y$ & $z$ \\
\hline
\seqnum{A004956} &0& 2 & 2 & 1 &
\seqnum{A081841} &0& 0 & 3 & 2 \\
\seqnum{A007066} &0& 1 & 2 & 3 &
\seqnum{A081840} &0& 0 & 3 & 4 \\
\seqnum{A045412} &0& 3 & 1 & 3 &
\seqnum{A081842} &0& 0 & 4 & 3 \\
\seqnum{A064437} &0& 1 & 3 & 2 &
\seqnum{A081839} &0& 0 & 4 & 5 \\
\seqnum{A080578} &0& 1 & 1 & 3 &
\seqnum{A081843} &0& 0 & 5 & 4 \\
\seqnum{A080579} &0& 1 & 1 & 4 &
\seqnum{A080578} &0& 1 & 1 & 3 \\
\seqnum{A080580} &0& 1 & 2 & 4 &
\seqnum{A080579} &0& 1 & 1 & 4 \\
\seqnum{A080590} &0& 1 & 3 & 4 &
\seqnum{A007066} &0& 1 & 2 & 3 \\
\seqnum{A080600} &0& 4 & 4 & 3 &
\seqnum{A080580} &0& 1 & 2 & 4\\
\seqnum{A080652} &0& 2 & 3 & 2 &
\seqnum{A064437} &0& 1 & 3 & 2 \\
\seqnum{A080667} &0& 3 & 4 & 3 &
\seqnum{A080590} &0& 1 & 3 & 4  \\
\seqnum{A080903} &0& 1 & 4 & 2 &
\seqnum{A080903} &0& 1 & 4 & 2 \\
\seqnum{A081834} &0& 1 & 4 & 3 &
\seqnum{A081834} &0& 1 & 4 & 3 \\
\seqnum{A081835} &0& 1 & 5 & 4 &
\seqnum{A081835} &0& 1 & 5 & 4 \\
\seqnum{A081839} &0& 0 & 4 & 5 &
\seqnum{A004956} &0& 2 & 2 & 1\\
\seqnum{A081840} &0& 0 & 3 & 4 &
\seqnum{A080652} &0& 2 & 3 & 2 \\
\seqnum{A081841} &0& 0 & 3 & 2 &
\seqnum{A045412} &0& 3 & 1 & 3 \\
\seqnum{A081842} &0& 0 & 4 & 3 &
\seqnum{A080667} &0& 3 & 4 & 3 \\
\seqnum{A081843} &0& 0 & 5 & 4 &
\seqnum{A080600} &0& 4 & 4 & 3 \\
\seqnum{A086377} &1& 1 & 3 & 2 &
\seqnum{A086377} &1& 1 & 3 & 2\\
\seqnum{A086398} &1& 1 & 4 & 2 &
\seqnum{A086398} &1& 1 & 4 & 2 \\
\hline
\end{tabular}

\caption{Table of $(j,x,y,z)$-hiccup sequences that are recognized as
such in the OEIS. 
Almost all parameters are $j=0, x=1$, and $|y-z|=1$. The left-hand
panel is ordered according to the OEIS number. The right-hand panel is
lexicographic in the parameters.
All sequences with the exception of the first two were entered by Benoit Cloitre over
the course of 2003.}\label{tbl:1}
\end{table}

A hiccup sequence is trivial if $x=z=1$. 
Consider, for example, the $(0,1,y,1)$-hiccup sequence for an arbitrary $y>1$. 
What is $a(2)$? At index~$2$, the only number already in the 
sequence is $a(1)=1$ and the next number is $a(3)=3$, etc. The~$y$
plays no role.
This problem does not arise if $x>1$,
as is illustrated by the first entry in Table~\ref{tbl:1}. It
is the only hiccup sequence with $z=1$ in the OEIS.

\medbreak

The literature on hiccup sequences is limited. Apart from Cloitre's recent
paper, the only study is
 by Bosma, Dekking, and Steiner~\cite{bosma2018}. It concerns
the $(1,1,3,2)$-sequence. Its
first few entries are
\[
1, 4, 6, 8, 11, 13, 16, 18, 21, 23, 25, 28, 30, 33, 35, 37, 40, 42, 45, 47,\ldots
\]
Bosma, Dekking, and Steiner proved that there are different ways to characterize this sequence,
that at first sight seem to be unrelated.
The aim of our paper is to generalize their remarkable result to hiccup sequences
that are strictly increasing.

\begin{theorem}[Bosma--Dekking--Steiner]\label{thm:1}
The sequence {\seqnum{A086377}} can be characterized in the following equivalent ways:
\begin{enumerate}[label=\textnormal{(\arabic*)}, leftmargin=1.2cm]
    \item \textbf{\textup{Hiccup sequence}} -- Defined recursively by
    \[
    a(n) = 
    \begin{cases}
        a(n-1) + 2, & \text{if } n-1 \notin \{a(k)\}_{k < n}, \\[4pt]
        a(n-1) + 3, & \text{if } n-1 \in \{a(k)\}_{k < n},
    \end{cases}
    \]
    with initial condition $a(1) = 1$.

    \item \textbf{\textup{Morphic characterization}} -- The positions of $1$'s in the fixed point of the morphism
    \[
    \begin{aligned}
    0 &\mapsto 10, \\
    1 &\mapsto 100.
    \end{aligned}
    \]

    \item \textbf{\textup{Beatty sequence representation}} --
    \[
    a(n) = \left\lfloor (1 + \sqrt{2})n - \frac{\sqrt{2}}{2} \right\rfloor.
    \]

    \item \textbf{\textup{Iterated function sequence}} --
    \[
    a(n) = \left\lfloor r_n + \frac{1}{2} \right\rfloor,
    \]
    where $r_1 = \dfrac{4}{\pi}$ and
    \[
    r_{n+1} = f_n(r_n), \quad \text{with } f_n(x) = \frac{n^2}{x + 1 - 2n}.
    \]
\end{enumerate}
\end{theorem}

We will refer to Theorem~\ref{thm:1} as the\textit{ BDS theorem}. The sections in our
paper are numbered according to the characterizations of this theorem. In each
section we extend the BDS theorem to more general hiccup sequences.
\medbreak

We highlight some of the scattered comments on hiccup sequences that can be found on the OEIS.
The $(0,1,2,3)$-hiccup sequence is entry \seqnum{A007066} with first few numbers
\[
1, 4, 7, 9, 12, 15, 17, 20, 22, 25, 28, 30, 33, 36, 38, 41, 43,\ldots
\]
It is entered in the OEIS as a Beatty sequence
\[
a(n)=\left\lceil (n-1)\cdot\phi^2 + 1\right\rceil,
\]
where $\phi=\frac{1+\sqrt{5}}2$ is the golden mean. 
Except for its initial entry, 
it is equal to the upper Wythoff sequence \seqnum{A001950} plus two.
It also corresponds to the positions of $1$ of the fixed
point of $0\mapsto 010,\ 1\mapsto 10$ that starts with $1$.
These facts follow from Fraenkel's characterizations
of the upper Wythoff sequence~\cite{fraenkel1982}.

The $(1,1,4,2)$-hiccup sequence \seqnum{A086398} starts out with
\[
1, 5, 7, 9, 11, 15, 17, 21, 23, 27, 29, 33, 35, 37, 39, 43, 45, 49,\ldots.
\]
Kimberling asked if it corresponds to the positions of $1$ in the
unique fixed point of $0\mapsto 10,\ 1\mapsto 1000$. It is not so hard to
see that this is true. The fixed point 
\[\omega=1000101010100010100010\cdots\]
satisfies
\[\omega=\sigma(\omega)=\sigma(1)\sigma(0)\sigma(0)\sigma(0)\sigma(1)\sigma(0)\cdots.\]
The $n$-th $1$ in the sequence corresponds to the $n$-th $\sigma(d)$ for $d=0$ or $d=1$
in $\sigma(\omega)$. It is $4$ ahead of the previous $1$ exactly if the $(n-1)$-th
digit of $\omega$ is a $1$. In other words, it is $4$ ahead if and only if $n-1$ is in
the sequence. 

Kimberling also asked if $-1 < n\left(1 + \sqrt{3}\right) - a(n) < 4$, which we will
confirm in Section~3. There is no
Beatty sequence readily available, unlike in the previous examples.
Since all numbers
in the sequence are odd, we can convert it by $a(n)\mapsto b(n)=(a(n)+1)/2$, which gives \seqnum{A026363}.
Here, Kimberling asks if $-1 < n\left(1+\sqrt{3}\right)/2 - b(n) < 2$, which is a slightly
weaker bound.

The $(0,1,2,4)$-hiccup sequence is entry \seqnum{A080580} in the OEIS.
\[
1, 5, 9, 13, 15, 19, 23, 27, 29, 33, 37, 41, 43, 47, 49,  53, 57, 61,\ldots
\]
Joe Slater proves in a comment on this sequence that it coincides with the positions of $1$ in the 
(non-unique) fixed
point of $0\mapsto 0010,\ 1\mapsto 10$ that starts with $1$. By the same argument, the $n$-th $1$ in this
sequence is $4$ ahead of the previous $1$ exactly if $n$ is not in the sequence. Otherwise,
it is ahead by $2$. We will call this Slater's construction in this paper. 

If we change the $x$ parameter from $2$ to $1$ we get the $(0,1,1,4)$-hiccup sequence,
which is the previous entry \seqnum{A080579}
\[
1, 5, 9, 13, 14, 18, 22, 26, 27, 31, 35, 39, 40, 41, 45, 49, 53, 54, \ldots
\]
in the OEIS. Slater's argument works for the fixed point of $0\mapsto 0001,\ 1\mapsto 1$ that starts
with $10$. If we change $z$ from $4$ to $3$ we get entry \seqnum{A080578}, the $(0,1,1,3)$-hiccup sequence,
which by Slater's construction is the fixed point of $0\mapsto 001,\ 1\mapsto 1$ that starts with $10$,
see also~\cite{ruskey2006}. Interestingly, this sequence is morphic but not automatic, as shown in
example 16 of~\cite{allouche2022}. 

Most hiccup sequences are stand alone, but some have been linked to other seemingly unrelated
sequences. In particular, 
\seqnum{A040412} has been linked to bottom-up search and
\seqnum{A080578} has been linked to meta-Fibonacci sequences.
More details on such connections can be found in Cloitre's recent paper~\cite{cloitre2025}.

\section{Morphisms for hiccup sequences}

An increasing sequence $a(n)$ corresponds to a binary sequence $b(n)$, such that $b(n)=1$
if and only if $n=a(k)$ for some $k$. 
If we interpret the sequence as a set $A=\{a(n)\}\subset \mathbb N$, then    
the \emph{characteristic sequence} is equal to $1$
at position $a$ if $a\in A$ and it is $0$ otherwise.
We will show that the characteristic sequence of a hiccup
sequence is morphic. 
A \emph{morphism} or \emph{substitution} on an alphabet is a formal rule that replaces each letter of the alphabet with a finite word over that same alphabet. 
\begin{definition}\label{def:1}
Let $\Sigma$ and $\Delta$ be finite alphabets. A sequence $(x_n)_{n \geq 1} \in \Delta^{\mathbb{N}}$ is called a \emph{morphic sequence}, see~\cite[Ch. 7]{alloucheshallit}, if there exist:

\begin{itemize}
  \item a morphism $\phi : \Sigma \to \Sigma^*$ such that $\sigma$ has a fixed point $w = \phi^\infty(a)$ for some $a \in \Sigma$, and
  \item a coding $\pi : \Sigma \to \Delta$ (i.e., a letter-to-letter morphism),
\end{itemize}

such that
\[
x_n = \pi(w_n) \quad \text{for all } n \geq 1.
\]
\end{definition}
\noindent We say that an increasing sequence $a(n)$ is morphic if its characteristic
sequence is morphic.

\begin{lemma}\label{lem:01yz}
    A $(0,1,y,z)$-hiccup sequence is morphic.
\end{lemma}
\begin{proof}
    We already went over this argument for $z>y$ in the previous section. 
    Define 
    \begin{equation}\label{eq:4}
    \phi(1)=10^{y-1},\ \phi(0)=0^{z-y}10^{y-1}
    \end{equation} 
    and let $(w_i)$ be the
    fixed point with $w_i\in\{0,1\}$ and $w_1w_2=10$. That second letter~$0$
    takes care of case $y=1$.
    Since the sequence is a fixed point, we have $(w_i)=(\phi(w_i))$. 
    The $n$-th $1$ is in $\phi(w_n)$. It is 
    preceded by $y-1$ zeros if and only if the first letter of $\phi(w_n)$
    is $1$. Otherwise, it is preceded by $z-1$ zeros. 
    Therefore, if $a(n)$ is the position of the $n$-th zero, then 
    $a(n)-a(n-1)=y$ if $n$ is in the sequence and 
    $a(n)-a(n-1)=z$ if it is not. Thus, this construction yields a hiccup sequence with the desired properties. Its initial entry is $a(1)=1$. 

    If $z<y$, 
    we add a letter $b$ to the alphabet, and define a morphism on $\{b,0,1\}$
    by 
    \begin{equation}\label{eq:5}
b\mapsto b0^{z-1},\ 0\mapsto 10^{z-1},\ 1\mapsto 0^{y-z}10^{z-1}.
    \end{equation}
    Let $(w_i)$ be the unique fixed point with initial letters $b0$. 
    If we count $b$ as a $1$, i.e., apply a coding $b\mapsto 1$,
    then every code word contains one $1$.  
    By the same argument as before, the $n$-th $1$ is preceded by $y-1$ zeros if 
    $w_n=1$ and by $z-1$ zeros if $w_n=0$.
\end{proof}

The additional letter $b$ in the proof of the lemma, which counts as a special $1$ so to speak, is needed to get the 
right initial condition $a(1)=1$. The same trick works if $a(1)>1$.

\begin{lemma}\label{lem:oxyz}
    A $(0,x,y,z)$-hiccup sequence with $x>1$
    is morphic.
\end{lemma}
\begin{proof}
    We add a letter $b$ to the alphabet and consider the fixed point that
    starts with~$b$.
    If $y<z$ we define the morphism  
    \begin{equation}\label{eq:6}
    b\mapsto b0^{x-2}10^{y-1},\ 0\mapsto 0^{z-y}10^{y-1},\ 1\mapsto 10^{y-1}.
    \end{equation}
    Note that $b$ can be removed from this morphisms if $x=z-y+1$,
    when the fixed point that starts with $0$ produces the hiccup sequence. If $y>z$ we define
    \begin{equation}\label{eq:7}
    b\mapsto b0^{x-2}10^{z-1},\ 0\mapsto 10^{z-1},\ 1\mapsto 0^{y-z}10^{z-1}.
    \end{equation}
    Now code $b\mapsto 0$. 
\end{proof}

Note that if we remove $b$ from the morphism in Equation~\ref{eq:6}, this gives the morphism
in Equation~\ref{eq:4}. Similarly, if we replace $b\mapsto b0^{x-2}10^{z-1}$ by $b\mapsto b0^{z-1}$
in Equation~\ref{eq:7},
i.e., remove the factor $0^{x-2}1$ from the image of $b$, then we get Equation~\ref{eq:5}. 
With these morphisms for $j=0$ in hand, we can find sequences 
in the OEIS that have not yet been recognized as hiccup sequences. For instance, if $x=2, y=2, z=3$
then we get \seqnum{A026356}. This sequence is equal to \seqnum{A007066}, except for the first entry,
which is the hiccup sequence with $x=1, y=2, z=3$. 
The morphism
\begin{equation}\label{eq:010}
0\mapsto 010,\ 1\mapsto 10
\end{equation}
generates both of these sequences. 
\seqnum{A026356} is the fixed point starting with $0$, and \seqnum{A007066} is the one starting with $1$,
see~\cite[Exercise 6.1.25]{fogg}. That explains why the two sequences are the same, except for their initial entries.

The morphisms in these lemmas are not unique, there are other morphisms that code the same sequence. 
For instance, our morphism for the $(0,2,4,2)$-hiccup sequence in Equation~\ref{eq:7}
is given by $b\mapsto b10,\ 0\mapsto 10,$ $1\mapsto 0010$. This morphism can be simplified to
 $0\mapsto 01,$ $1\mapsto 0001$, which by Slater's construction produces the same fixed point 
 without using the special letter $b$. The $(0,2,4,2)$-hiccup sequence is entry \seqnum{A284753} in the OEIS. It is not yet listed as a hiccup sequence. 
 
 There is a degree of freedom in the morphism, by a cyclic permutation of the
 image words of $0$ and $1$. If
 we replace the morphism in Equation~\ref{eq:6} by $b\mapsto b0^{x-2}10^{y-2},\ 0\mapsto 0^{z-y+1}10^{y-2},\ 1\mapsto 010^{y-2},$ assuming that $y\geq 2$, then we get the same fixed point.
 Similarly, we can replace the morphism in Equation~\ref{eq:7} by
 $b\mapsto b0^{x-2}10^{z-2},\ 0\mapsto 010^{z-2},\ 1\mapsto 0^{y-z+1}10^{z-2},$
 if $z\geq 2$. This is exactly how we got that morphism $0\mapsto 01,\ 1\mapsto 0001$ above.
 If the parameters allow it, it is possible to get rid of the special letter $b$ by
 iterating the cyclic permutation. A sequence is \emph{purely morphic} if
 there is no need for the coding $\pi$ in Definition~\ref{def:1}.

\begin{lemma}\label{lem:pure}
    A $(0,x,y,z)$-hiccup sequence is purely morphic if $z-y+1\leq x\leq z$.
\end{lemma}
\begin{proof}
    A cyclic permutation of the code words for $0$ and $1$ gives the morphism
    \begin{equation}\label{eq:6b}
    0\mapsto 0^{x-1}10^{z-x},\ 1\mapsto 0^{x+y-z-1}10^{z-x}.
    \end{equation}
\end{proof}

\begin{lemma}\label{lem:1xyz1}
    A $(1,x,y,z)$-hiccup sequence with $x\geq 1$
    is morphic. It is purely morphic if $x=1$ and $y>1$.
\end{lemma}
\begin{proof}
    First suppose $x>1$.
    We add the letter $b$ and we define
    \[b\mapsto b0^{x-2}10^{z-1},\ 0\mapsto 10^{z-1},\ 1\mapsto 10^{y-1}.\]
    Let $(w_i)$ be the unique fixed point that starts with $b$, which
    is coded to $0$. The $n$-th $1$ in the sequence is in $\sigma(w_n)$.
    It is preceded by $y-1$ zeros if $w_{n-1}=1$, i.e., if $n-1$ is
    in the sequence, and otherwise by $z-1$ zeros.

    If $x=1$ and $y=1$,
    then we get the trivial sequence $a(n)=n$.
    If $x=1$ and $y>1$ then there is no need for the letter $b$, and we can
    take $ 0\mapsto 10^{z-1},\ 1\mapsto 10^{y-1}$.
\end{proof}

Some of these morphisms correspond to entries in the OEIS that have not been recognized
as hiccup sequences.
For instance, the $(1,1,2,1)$-hiccup sequence is the fixed point of the Fibonacci morphism 
$0\mapsto 1,\ 1\mapsto 10$, which gives the lower Wythoff sequence~\seqnum{A000201}. The
$(1,1,3,1)$-hiccup sequence is equal to \seqnum{A003156}, and the $(1,1,2,3)$-hiccup sequence
is \seqnum{A026352}. The result is not sharp. There are more choices of
parameters for which hiccup sequences are purely morphic.
For instance, the morphism $0\mapsto 010,\ 1\mapsto 01$, which is the reversal of the morphism in
\ref{eq:010}, generates
the $(1,2,2,3)$-hiccup sequence by Slater's argument. It is the upper Wythoff sequence \seqnum{A001950}.

\begin{lemma}\label{lem:00yz}
    Let $b(n)$ be the $(0,z-1,y,z)$-hiccup sequence for $y,z>1$, and let 
    \[a(n+1)=b(n)+1,\ a(1)=0.\]
    Then $a(n)$ is the $(0,0,y,z)$-hiccup sequence.
    In particular, $a(n+1)$ is morphic.
\end{lemma}

For instance, \seqnum{A081841} is equal to \seqnum{A064437} plus one.

\begin{proof}
    Since $y,z>1$ we have that $b(n+1)\geq b(n)+2$ for all $n$ and $b(1)=z-1\geq 1$.
    It follows that $b(n)>n$ for $n>1$.
    By definition, $a(2)=b(1)+1=z$, which is the correct value for the $(0,0,y,z)$-hiccup sequence.    
    For $n>1$ we have \[a(n+1)=b(n)+1=b(n-1)+y+1=a(n)+y\] if $n\in \{b(k)\}_{k<n}$, and otherwise 
    \[a(n+1)=b(n)+1=b(n-1)+z+1=a(n)+z.\] 
    Since $a(k)=b(k-1)+1$ and $a(1)=0$, it follows that $a(n+1)=a(n)+y$ if and only if $n+1\in\{a(k)\}_{k<n+1}$,
    and it is $a(n+1)=a(n)+z$ otherwise. 

    The sequence $b(n)$ corresponds to the positions of $1$ in a coded fixed point of a morphism.
    This morphism has an alphabet of two or three letters, depending on
    the parameters. To convert this to a morphism for $a(n+1)=b(n)+1$, add a letter $c$ and add
    the substitution $c\mapsto cd$ to the morphism, 
    where $d$ is the initial letter of the fixed point for $b(n)$.
    This is the morphism for $a(n+1)$.
    Add $c\mapsto 0$ to the coding.
\end{proof}

\begin{lemma}\label{lem:1xyz2}
    Let $a(n)$ is the $(j,x,y,z)$-hiccup sequence for $j>0$ and let
    \[b(n)=a(n)+j.\]
    Then $b(n)$ is the $(0,x+j,y,z)$-hiccup sequence. 
\end{lemma}

There are two entries in Table~\ref{tbl:1} that have $j=1$.
Hiccup sequence \seqnum{A086377}, which was considered in \cite{bosma2018}, is equal to \seqnum{A080652} minus one. 
The other entry \seqnum{A086398} is equal to \seqnum{A284753} minus one.

\begin{proof}
    By definition, $b(n)=a(n)+j=a(n-1)+y+j$ if $n-j\in\{a(k)\}_{k<n}$ and it is $a(n-1)+z+j$ if not.
    Now $n-j\in\{a(k)\}_{k<n}$ is equal to $n\in\{b(k)\}_{k<n}$, which defines the $(j,z,y,z)$-hiccup sequence.
\end{proof}

\begin{theorem}\label{thm:main2}
    Hiccup sequences are morphic. A $(0,x,y,z)$-sequence is purely morphic
    if $z-y+1\leq x\leq z$ or if $x=1$.
\end{theorem}
\begin{proof}
    This is a consequence of the preceding lemmas. The sequence is purely
    morphic if $z-y+1\leq x\leq z$ according to Lemma~\ref{lem:pure}. If
    $x=1$ and $y>z$, then $z-y+1\leq x\leq z$. If $x=1$ and $y<z$, then
    the sequence is purely morphic according to Equation~\ref{eq:4}.
\end{proof}

This extends the second characterization in the BDS theorem. We conclude with some remarks
on the morphisms that we encountered in this section. They are defined on the alphabet 
$\{b,0,1\}$ and the letter $b$ only occurs as the initial letter of the fixed point.
The \emph{adjacency matrix} of a morphism registers the number of letters in the
substitution words. For instance, the morphism $0\mapsto 0100,\ 1\mapsto 100$ for
the $(0,1,3,4)$-hiccup sequence has
adjacency matrix $\begin{bmatrix}
3 & 2 \\
1 & 1
\end{bmatrix}
$.
For the morphisms that we found in this section, the restriction to $\{0,1\}$ 
has adjacency matrix equal to 
\begin{equation}\label{eq:adjacency}
\begin{bmatrix}
z-1 & y-1 \\
1 & 1
\end{bmatrix}.
\end{equation}
If an iterate of the adjacency matrix has all entries $>0$,
then the morphism is \emph{primitive}. 
In particular, our morphisms are \emph{primitive} if $y>1$,
when restricted to the alphabet $\{0,1\}$. 
If $y=1$, then the morphism is not primitive.

\section{Beatty sequences and beyond}

By the results of the previous section, in particular Lemma~\ref{lem:1xyz2}, we may
restrict our attention to $(0,x,y,z)$-hiccup sequences with $x,y,z$ in $\mathbb Z_{\geq 1}$.
In fact, we will restrict our attention even further and consider only 
hiccup sequences that are purely morphic.
\medbreak
The semigroup of \emph{Sturmian morphisms}, see \cite[p. 72] {Berstel}, is generated by 
the three substitutions \(L\), \(R\), and \(E\):
\hspace{-0.2cm}
\[
E:
\begin{cases}
0 \mapsto 1 \\
1 \mapsto 0
\end{cases}
\quad
L: 
\begin{cases}
0 \mapsto 01 \\
1 \mapsto 0
\end{cases}
\quad
R:
\begin{cases}
0 \mapsto 10 \\
1 \mapsto 0
\end{cases}
\]
\noindent
For instance, the square
$R^2$ produces the morphism $0\mapsto 010,\ 1\mapsto 10$ in Equation~\ref{eq:010} 
which has two fixed points. One corresponds to \seqnum{A026356}
and the other corresponds to \seqnum{A007066}. The square $L^2$ produces the
reversal of $R^2$ given by $0\mapsto 010,\ 1\mapsto 01$. Its fixed point is the $(1,2,2,3)$-hiccup
sequence, which is better known as the upper Wythoff sequence \seqnum{A001950}.

It follows from Theorem~\ref{thm:main2}, a $(0,x,y,z)$-hiccup sequence is purely
morphic if $x\leq z$ and $|z-y|=1$. We prove that the morphisms that generate
these sequences are Sturmian.

\begin{lemma}
    A $(0,x,y,z)$-hiccup sequence is Sturmian if $x\leq z$ and $|z-y|=1$.
\end{lemma}
\begin{proof}
    First suppose that $y=z+1$.
    The morphism $0\mapsto 0^{x-1}10^{z-x},\ 1\mapsto 0^{x}10^{z-x}$ 
    generates the $(0,x,z+1,z)$-hiccup sequence according to
    Equation~\ref{eq:6b}. We prove that the morphism is Sturmian. 
    The morphisms
    \[
  G\colon  \begin{cases}
0 \mapsto 0 \\
1 \mapsto 01
\end{cases}
\text{ }
\tilde  G\colon  \begin{cases}
0 \mapsto 0 \\
1 \mapsto 10
\end{cases}
\text{ }
H\colon
\begin{cases}
0 \mapsto 1 \\
1 \mapsto 01
\end{cases}
    \]
    are Sturmian. They are $G=L\circ E$, $\tilde G=R\circ E$,
    and $H=E\circ \tilde G$.
    Now $G^{x-1}\circ \tilde G^{z-x}\circ H$ produces the required morphism.    

    Now suppose $y=z-1$. If $x\geq 2$, then 
    $0\mapsto 0^{x-1}10^{z-x},\ 1\mapsto 0^{x-2}10^{z-x}$ 
    generates the $(0,x,z-1,z)$-hiccup sequence, according to
    Equation~\ref{eq:6b}. It is equal to
    $G^{x-2}\circ \tilde G^{z-x}\circ H\circ E$.
    Recall that we require $z>1$ if $x=1$. If $x=1$, then
    $0\mapsto 010^{z-2},\ 1\mapsto 10^{z-2}$ generates the
    $(0,1,z-1,z)$-hiccup sequence, starting from $1$, according
    to Equation~\ref{eq:4}. 
    This is the same morphism used for the
    $(0,2,z-1,z)$-hiccup sequence, but starting from
    the other letter.
\end{proof}

There is a well-known correspondence between Sturmian substitutions and the sequences
\[
S_{\alpha,\beta}(n)=\lfloor \alpha (n+1) +\beta\rfloor - \lfloor \alpha n+\beta\rfloor,
\]
which are known as \emph{Sturmian words} or \emph{mechanical words}, see \cite[p53]{Berstel}.
The values of $\alpha$ and $\beta$ can be determined from the following transformations
\[
T_E(x,y)=\left(1-x,1-y\right),\ 
T_L(x,y)=\left(\frac{1-x}{1-2x},\frac{1-y}{1-2x} \right),\ 
T_R(x,y)=\left(\frac{1-x}{1-2x},\frac{2-x-y}{1-2x} \right).
\]
A composition of $E,L,R$ corresponds to a composition of these
transformations, and $(\alpha,\beta)$ is a fixed point of the resulting composite transformation, 
see~\cite[p73]{Berstel} and~\cite{bosma2018}.
As an example, we compute $\alpha$
and $\beta$ for the fixed point of the morphism $0\mapsto 010,\ 1\mapsto 10$.
The morphism has adjacency matrix
\[
M = \begin{bmatrix}
2 & 1 \\
1 & 1
\end{bmatrix}
\]
which has largest eigenvalue $\lambda=\frac{3+\sqrt 5}2$. 
Now $\alpha$ is equal to the density   
of $1's$ in the fixed point of the morphism,
and a minor calculation gives $\alpha=\frac{3-\sqrt{5}}2$.
The fixed point $(\alpha,\beta)$ for $T_{R^2}=T_RT_R$ satisfies $\beta=1-\alpha$.
This is one of the special cases in which rounding up or down makes a difference
for the mechanical word. The
rounded-up sequence
corresponds to the fixed point beginning with $0$, while the rounded-down one
begins with $1$.

Mechanical words can be converted to Beatty sequences~\cite{beatty}.
The following result and its proof is almost identical to~\cite[Lemma 1]{bosma2018}.
\begin{lemma} Let $\alpha > 1$ be irrational, and let $(s_n)_{n \geq 1}$ be given by
the rounded-up mechanical sequence
\[ s_n=\lceil (n+1)\alpha  +\beta\rceil - \lceil n\alpha +\beta\rceil, \]
for some real number $\beta$ with $\alpha + \beta \geq 1$.
Then the rounded-down Beatty sequence
\[
\left\lfloor \frac k\alpha -\frac \beta \alpha \right\rfloor\ \text{ for }k\in \mathbb N
\]
corresponds to the sequence of positions of 1 in $(s_n)$. Similarly, a rounded-down
mechanical sequence corresponds to a rounded-up Beatty sequence.
\end{lemma}
\begin{proof} 
Since $\alpha+\beta\geq 1$ we have that $\frac k\alpha -\frac \beta \alpha\geq 1$ if $k\in\mathbb N$.
It follows that the elements of our Beatty sequence are all natural numbers.
\[
\begin{aligned}
\exists k \geq 1 : n = \left\lfloor \frac k\alpha -\frac \beta \alpha \right\rfloor 
&\Longleftrightarrow \exists k \geq 1 : n \leq \frac k\alpha -\frac \beta \alpha < n + 1 \\
&\Longleftrightarrow \exists k \geq 1 : {n\alpha+\beta} \leq k < {(n+1)\alpha+\beta} \\
%&\text{\small Here we needed the condition that }\beta\not\in \mathbb N-\mathbb N \alpha\\
&\Longleftrightarrow \exists k \geq 1 : 
\left\lceil{n\alpha+\beta} \right\rceil = k  
\text{ and } 
\left\lceil {(n+1)\alpha+\beta} \right\rceil = k + 1 \\
&\text{\small because }\alpha<1\\
&\Longleftrightarrow 
\left\lceil {(n+1)\alpha+\beta} \right\rceil 
- \left\lceil {n\alpha+\beta} \right\rceil = 1 \\
&\Longleftrightarrow s_n = 1.
\end{aligned}
\]
The same argument gives that $ n = \left\lceil \frac k\alpha -\frac \beta \alpha \right\rceil $
if and only if $\left\lfloor {n\alpha+\beta} \right\rfloor 
- \left\lfloor {(n-1)\alpha+\beta} \right\rfloor = 1$.
\end{proof}

For instance, \seqnum{A007066} is given by
\[
\left\lceil \frac{k}{\alpha}-\frac{1-\alpha}{\alpha}\right\rceil 
\]
while \seqnum {A026356} is given by
\[
\left\lfloor \frac{k}{\alpha}-\frac{1-\alpha}{\alpha}+1\right\rfloor,
\]
for $\alpha=\frac{3-\sqrt 5}{2}$.
These lemmas lead to the following generalization of the third characterization of the BDS theorem,
which was recently proved by Cloitre. He provides many more concrete expressions for hiccup
sequences that are Beatty sequences.

\begin{theorem}[Cloitre, \cite{cloitre2025}]\label{thm:main3}
    A hiccup sequence is a Beatty sequence if $x\leq z$ and $y>1$ and $|y-z|=1$.
\end{theorem}

Out of the 21 hiccup sequences in Table~\ref{tbl:1}, there are 16
that satisfy this condition.
Out of the remaining 5 sequences, 3 have parameter $y=1$, namely
$(0,1,1,3), (0,1,1,4),$ and $(0,3,1,3)$. The morphisms that generate these
sequences are not primitive, as we saw Equation~\ref{eq:adjacency} above.
We already mentioned that the $(0,1,1,3)$-hiccup sequence \seqnum{A080578}
is a morphic sequence that is not automatic. It is example 16 in
\cite{allouche2022}. By the same argument, it is possible to show
that $(0,x,1,z)$-hiccup sequences are not automatic.
\medbreak
From the remaining two hiccup sequences that are non-Beatty
we consider the $(0,2,4,2)$-hiccup sequence \seqnum{A284753}, which
is a relatively recent addition to the OEIS. 
It is equal to $1+\seqnum{A086398}$ by Lemma~\ref{lem:1xyz2}, and has some interesting
properties. In the comments on this sequence on the OEIS, Dekking proves
that it is equal to 2 times \seqnum{A026363}. Kimberling conjectures
that $-2<(1+\sqrt 3)n-a(n)<3$, and Dekking remarks
that $(1+\sqrt 3)n-a(n)$ is bounded by a result
of Adamczewski~\cite{adamczewski2004}, where $a(n)$ is \seqnum{A284753}. We explore the sequence with the assistance of the automatic
theorem prover \texttt{Walnut}~\cite{mousavi, shallit2022}. 

Schaeffer, Shallit, and Zorcic~\cite{schaeffer2024} proved that the first-order logical theory of quadratic Beatty sequences with addition is decidable. We apply their ideas to study \seqnum{A284753}. This hiccup sequence is purely morphic and generated by $0\mapsto 01,\ 1\mapsto 0001$. We used Ollinger's licofage toolkit~\cite{licofage}, which is described in
Carton et al.~\cite{carton2025}, to convert the morphism to a Dumont-Thomas numeration
system for \texttt{Walnut}. The base of the numeration system corresponds to the lengths of
the iterated substitution words
\[
0\mapsto 01\mapsto 010001\mapsto 0100010101010001\mapsto \cdots,
\]
which are equal to $1,2,6,16,44,\ldots$, and satisfies the linear recurrence $B_{n+1}=2B_n+2B_{n-1}.$
Every number is represented by a unique word with digits $\{0,1,2,3\}$ that is accepted by the
automaton in Fig.~\ref{fig:1}. For instance, the number $39$ is represented by $1321$ in this
numeration system.
\begin{figure}
    \centering
    \includegraphics[width=0.5\linewidth]{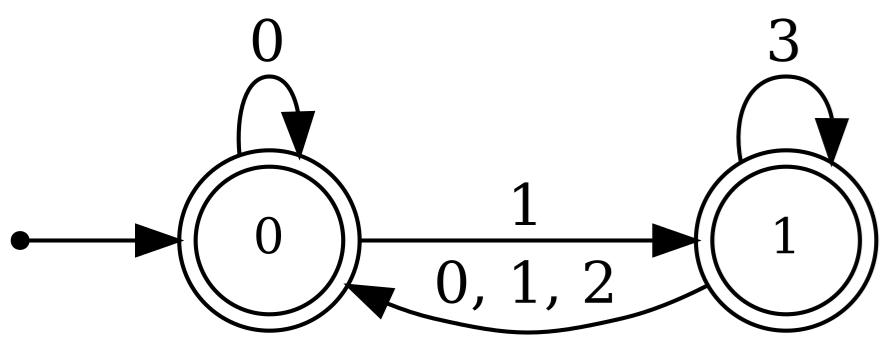}
    \caption{\small The msd expressions that are accepted in the Dumont-Thomas numeration system
    for the morphism $0\mapsto 01,\ 1\mapsto 0001$ have digits in $\{0,1,2,3\}$. This
    automaton has two accepting states. Transitions that are not included in the automaton are rejected. The expression that are accepted are
    of the form $\{0^*13^*\{0|1|2\}\}^*$, where $*$ is the Kleene star, which means that
    the expression can be repeated zero or more times.} 
    \label{fig:1}
\end{figure}
The first ten entries of the sequence are given in Table~\ref{tbl:2} in decimal and in
Dumont-Thomas representation. It is now easy to guess that the Dumont-Thomas representation
of the sequence is $(0w,w0)$ where $w$ runs over the admissible words with most initial 
digit $1$. This is in line with the results of Schaeffer, Shallit, and Zorcic, who proved
that such expressions describe Beatty sequences in Ostrowski numeration. It is a standard
exercise to define an automaton in \texttt{Walnut} that recognizes such expressions,
see \cite[p.106]{shallit2022}. We implemented it as \texttt{shift0123} and we named the
Dumont-Thomas numeration system as \texttt{msd\_dumthomabaaab}. It is a standard task for
\texttt{Walnut} to verify that the sequence is indeed of the form $(0w,w0)$ in this system.
\newline
\begin{table}[h!]
\centering
\caption{A table with the first ten entries of $(n,a(n)$ where $a(n)$ is sequence \seqnum{A284753}, in the standard decimal representation and in the Dumont-Thomas representation. }\label{tbl:2}
\begin{tabular}{|c|c||c|c|}
\hline
\multicolumn{2}{|c||}{\textbf{\ \ Decimal\ \ }} & \multicolumn{2}{|c|}{\textbf{Dumont-Thomas}} \\
\hline
$n$ & $a(n)$ & $\ \ \  n\ \ $ & $a(n)$  \\
\hline
1  & 2  & 1 &  10\\
2  & 6  & 10 & 100 \\
3  & 8  & 11 & 110 \\
4  & 10 & 12 & 120 \\
5  & 12 & 13 & 130 \\
6  & 16 & 100 &1000  \\
7  & 18 & 101 &1010  \\
8  & 22 & 110 &1100  \\
9  & 24 & 111 &1110  \\
10 & 28 & 120 &1200  \\
\hline
\end{tabular}
\end{table}

We first test if differences $a(n+1)-a(n)$ are equal to either 2 or 4. The test is carried
out within the Dumont-Thomas system with most significant digit (msd) first representation.
\smallbreak 
\noindent\texttt{eval test1 "?msd\_dumthomabaaab An,s,t (n>0 \& \$shift0123(n,s) \&
\newline \$shift0123(n+1,t)) => t=s+2|t=s+4";}
\smallbreak 
\noindent
We then test if the difference is $4$ if the index $n$ is in the sequence.
\smallbreak
\noindent\texttt{eval test2 "?msd\_dumthomabaaab An,s,t (n>0 \& \$shift0123(n,s) \& \newline Ek (\$shift0123(k,n+1)) \& \$shift0123(n+1,t)) => t=s+4";}
\smallbreak 
\noindent
Finally, we test that the difference is $2$ if the index is not in the sequence.
\smallbreak\noindent\texttt{eval test3 "?msd\_dumthomabaaab An,s,t (n>0 \& \$shift0123(n,s) \& \newline
(\~{}Ek (\$shift0123(k,n+1))) \& \$shift0123(n+1,t)) => t=s+2";}
\smallbreak 
\noindent
The sequence in \texttt{shift0123} passes the tests so that indeed 
\seqnum{A284753} is of the form $(0w,w0)$ in the Dumont-Thomas numeration system
generated by $0\mapsto 01,\ 1\mapsto 0001$.
\medbreak
In the introduction, we observed that the two sequences defined 
in Equations~\ref{eq:1} and~\ref{eq:2} have many common entries.
The following establishes the connection between these two hiccup sequences using simple \texttt{Walnut} checks.
\begin{lemma}
    We have \seqnum{A284753}$(n)=$\seqnum{A086398}$(n)+1$, and
    \begin{equation*}
        \text{\seqnum{A086398}}(n)=\begin{cases}
        \text{\seqnum{A080903}}(n)+2,&\text{if the DT representation
of $n$ ends with a $0$}\\
        \text{\seqnum{A080903}}(n), &\text{otherwise.}
    \end{cases}
    \end{equation*}
\end{lemma}
\begin{proof}
    With \texttt{Walnut}, we define a new synchronised sequence \texttt{a086398} that accepts one more than the term accepted in \texttt{shift0123} in parallel with $n$ as follows:
    \smallbreak\noindent\texttt{def a086398 "?msd\_dumthomabaaab Ek (\$shift0123(n,k) \& p+1=k)":}
    \smallbreak
    We can then verify that \texttt{a086398} indeed accepts $n$  and the terms of the $(1,1,4,2)$- hiccup sequence in parallel using the tests -- \texttt{test1},\texttt{test2}, and \texttt{test3} where \texttt{shift0123} replaced by \texttt{a086398}, and here we must modify our check -- if one less than the index is in the sequence.
    
    \noindent Now, for the conditional check, we define a DFA that accepts only the representations ending in a $0$, followed by defining a DFA accepting $n$ and \seqnum{A080903}$(n)$ in parallel.
    \smallbreak\noindent\texttt{reg endin0 msd\_dumthomabaaab "(0|1|2|3)*0";}
\smallbreak 
\smallbreak\noindent\texttt{def a080903 "?msd\_dumthomabaaab Es n>0 \& \$a086398(n,s) \& \\(\$endin0(n) => t+2=s) \& (\textasciitilde\$endin0(n) => t=s) ";}
\smallbreak 
Again, we will verify that \texttt{a086398} accepts all $n$ and the terms the $(0,1,4,2)$- hiccup sequence \seqnum{A080903} in parallel using the same tests -- \texttt{test1},\texttt{test2}, and \texttt{test3}, this time replacing \texttt{shift0123} with \texttt{a080903}.
\end{proof}

% It turns out that $a(n)=b(n)$ unless the Dumont-Thomas representation
% of $n$ ends with a $0$, in which case $a(n)=b(n)+2$. This fact can
% easily be checked with \texttt{Walnut}, since $a(n)+1$ is equal
% to \seqnum{A284753}.
Additionally, there is a simple characterization of \seqnum{A080903}.
\begin{lemma}\label{LEM:endin1}
    The numbers in \seqnum{A080903} are those that end with a $1$ in Dumont-Thomas numeration.
\end{lemma}
\begin{proof}
    We define a DFA accepting all the representations ending with a $1$ in \texttt{Walnut} with the following code:
    \smallbreak\noindent\texttt{reg endin1 msd\_dumthomabaaab "(0|1|2|3)*1";}
\smallbreak 
\noindent
This is followed by a simple check described in the statement of the lemma, which returns \texttt{TRUE}.
\smallbreak\noindent\texttt{eval test "?msd\_dumthomabaaab As (En (n>0 \& \$a080903(n,s)) <=> \$endin1(s))";}
\smallbreak 
\end{proof}

Now that we have a characterization of \seqnum{A284753} in terms of Dumont-Thomas
numeration, we can address Kimberling's conjectured bounds on the numbers in this
sequence.
Let $\lambda=1+\sqrt 3$ and $\bar\lambda=1-\sqrt 3$. The following is a Binet
type representation for the base $B_n$ of our numeration system.

\begin{lemma}
The basis $B_n$ of the Dumont-Thomas numeration system for $0\mapsto 01$,
$1\mapsto 0001$ satisfies
    \[B_n=\frac{\lambda^n-\bar\lambda^n}{2\sqrt 3}.\]
\end{lemma}
\begin{proof}
    These numbers satisfy the recursion $B_{n+1}=B_n+B_{n-1}$ with initial
    condition $B_0=0$ and $B_1=1$.
\end{proof}

\begin{lemma}
    \[B_{n+1}-\lambda B_n=\bar\lambda^{n}.\]
\end{lemma}
\begin{proof}
    This is a straightforward computation.
\end{proof}

In order to prove Kimberling's conjecture that $-3<a(n)-\lambda n<2$ we will run
into the problem of bounding
\[
\sum_{i\geq 1} d_i\bar\lambda^i
\]
over all admissible words $w=d_1d_2\cdots$ in least digit first format in our
Dumont-Thomas numeration system. Since $\bar\lambda$ is negative and less than
$1$ in absolute value, a straightforward lower bound is 
$3\sum_{i=1}^\infty \bar\lambda^{2i-1}$ and an upper bound is
$3\sum_{i=1}^\infty \bar\lambda^{2i}$, since digits are at most $3$.
These bounds are not sharp enough. We need to work a little harder. 

\begin{lemma}
    If $w=d_1d_2d_3d_4d_5d_6$ is any admissible word in lsd format in
    our Dumont-Thomas numeration system, then $\sum_{i=1}^6 d_i\bar\lambda^i$
    is minimized by $312121$. For a word in seven digits it is maximized by $0312121$.
\end{lemma}
\begin{proof}
    This is a finite problem that we solved numerically. 
    The minimum rounded to three decimals is $-2.424$. To understand why
    this is the minimizing word, without going into the computation, 
    notice that the odd powers $\bar\lambda^i$ are
    negative. We would like to maximize the $d_i$ for odd $i$. That is
    why $d_1=3$ is the logical candidate. If $d_i>1$ then $d_{i+1}$ is
    positive, which explains why $312121$ is the minimizer. The maximizer
    does the same thing, but for the even indexed digits. Its value
    rounded to three decimals is $1.775$.
\end{proof}

We write $m_6=-2.424$ for the minimum, rounded to three decimals, and
$M_7=1.775$ for the maximum.

\begin{theorem}\label{THM:kimbercon0242}
    The $(0,2,4,2)$-hiccup sequence $a(n)$ satisfies
    \[
    -3< a(n)-\lambda n< 2.
    \]
\end{theorem}
\begin{proof}
    Let $n=\sum d_iB_i$ be the Dumont-Thomas representation, for which
    we have $a(n)=\sum d_iB_{i+1}$. It follows that
    \[
    a(n)-\lambda n  =\sum d_i\bar\lambda_i^{i-1}.
    \]
    The powers $\bar\lambda^i$ alternate in sign and the digits are bounded by $3$.
    It follows that
    \[
        m_6+3\sum_{i=4}^\infty \bar\lambda^{2i-1} < a(n)-\lambda n < M_7+3\sum_{i=4}^\infty \bar\lambda^{2i}.
    \]
    These bounds rounded to three decimals are $-2.667$ for the minimum and $1.953$ for
    the maximum.
\end{proof}
This solves Kimberling's conjecture that $-2 < n\left(1 + \sqrt{3}\right) - \seqnum{A284753}(n) < 3$. A natural question arising from this result is whether the results of Schaeffer, Shallit, and Zorcic on Beatty sequences can be extended to other hiccup sequences. More specifically, is it true that hiccup sequences are synchronized
with respect to Dumont-Thomas numeration?
\newpage

\section{Iterated function sequences}

The fourth characterization of the BDS theorem 
is the most intriguing. It connects
Lambert's continued fraction expansion
of $\frac 4{\pi}$ to the $(1,1,3,2)$-hiccup
sequence. 
Lambert's expansion is given by
\[
{\frac 4 \pi } = 
1+
\cfrac{1^2}{3 + 
\cfrac{2^2}{5 + 
\cfrac{3^2}{7 + 
\cfrac{4^2}{\ddots}}}}, 
%\cfrac{5^2}{\ddots}}}}},
\]
which is a special case of his continued fraction expansion
of the arctangent, see~\cite{bauer2005}. It is well known that the continued
fraction converges, but that the rate of convergence is
not as fast as for regular continued fractions, see
\cite[Table 2.3]{jonesthron}. 
A shorthand notation for Lambert's continued fraction is
\[
\frac 4 \pi =1+\mathop{\mbox{\Large\bfseries K}}\limits_{j=1}^{\infty} \frac{j^2}{2j+1}
\]
The fractions $\frac{j^2}{2j+1}$ are called \emph{partial fractions}
and the truncated expansions $1+\mathop{\mbox{\Large\bfseries K}}\limits_{j=1}^{n} \frac{j^2}{2j+1}$
are called \emph{convergents}.
The initial convergents are \[\frac 1 1, \frac 4 3, \frac {24}{19}, \frac{204}{160},\ldots.\]
The numerators are \seqnum{A012244}, but the denominators are not included in the OEIS. 
The convergents are defined by cutting subsequent tails from the continued fraction. 
These tails are given by
\[
s_{n} =2n-1+\mathop{\mbox{\Large\bfseries K}}\limits_{j=n}^{\infty} \frac{j^2}{2j+1},
\] 
starting from $s_1=\frac 4 \pi=1+\frac 1{s_2}$.
Each next tail is connected to the previous tail by
\[
s_n=2n-1+\frac {n^2}{s_{n+1}}.
\]
The fourth characterization of the $(1,1,3,2,)$-hiccup sequence 
in the BDS theorem is that it is equal to 
$\left\lfloor s_n+\frac 12\right\rfloor$. It is suprising that
an iterative sequence that starts from $\frac 4\pi$ produces
the same result as a Beatty sequence with slope $1+\sqrt 2$.
Bosma, Dekking, and Steiner write that they were only convinced
after a numerical verification of the initial $130,000$ terms
of the sequence.

In their proof of the fourth characterization, Bosma, Dekking, and
Steiner consider the family of continued fractions of polynomial type given by
\begin{equation}\label{eq:Kk}
k-1+\mathop{\mbox{\Large\bfseries K}}\limits_{j=1}^{\infty} \frac{j^2}{(j+1)k-1},
\end{equation}
for natural numbers $k$. 
For $k=2$ we have the continued fraction of $\frac 4\pi$.
For $k=1$, Beukers~\cite{Beukers} found the following
elegant expression in terms of the golden mean $\phi=\frac{1+\sqrt 5}2$
and $\eta=\frac 1{1+\phi^2}$
\begin{equation}\label{eq:n2n}
\frac{\phi}
{\displaystyle \sum_{n=1}^\infty \dfrac 1{n+\eta}\left(\dfrac {-1}{\phi^2}\right)^n}
\;=\;
%\mathop{\mbox{\Large\bfseries K}}\limits_{j=1}^{\infty} \frac{j^2}{j}\;=\;
\cfrac{1^2}{1 + 
\cfrac{2^2}{2 + 
\cfrac{3^2}{3 + 
\cfrac{4^2}{\ddots}}}}.
\end{equation}
The numerators
of the convergents of this continued fraction
are \seqnum{A288952}, and appear in a study by Genitrini et al. on tree enumeration~\cite{genitrini2020}.
The tails of the continued fraction in Equation~\ref{eq:Kk} satisfy
\begin{equation}\label{eq:gaussmap}
s_n=kn-1+\frac {n^2}{s_{n+1}}.
\end{equation}
Expressing $s_{n+1}$ in terms of $s_n$ this is
\[
s_{n+1}=\frac {n^2}{s_{n}-kn+1}.
\]
In the final line of their
paper, Bosma, Dekking, and Steiner conjecture a remarkable relation
between this iterative sequence and a Beatty sequence that has 
slope $\alpha_k$ such that $\alpha_k^2=k\alpha_k+1$.
Such a quadratic number is known as a 
\emph{metallic mean}~\cite{metallic}.

\begin{conjecture}[Bosma, Dekking, and Steiner]\label{con:bds}
Let $\alpha_k$
be the $k$-th metallic mean and
let $s_1$ be the unique positive real number such that
all $s_{n+1}=\frac{n^2}{s_n-kn+1}$
are positive, where we suppress $k$ in our notation
of the sequence $s_n$. The rounded-down sequence  $\lfloor s_n\rfloor$
is equal to the Beatty sequence \[\left\lfloor \alpha_k.n-\frac{1+\alpha_k}{2\alpha_k-k} \right\rfloor.\]
\end{conjecture}
\vspace{0.5cm}
These are not hiccup sequences, and none of these sequences appear in the OEIS.
It is shown in \cite{bosma2018} that $s_n-n\alpha-\beta$ is positive
and converges to zero with order $O\left(\frac 1n\right)$,  where
we write $\alpha$ for the $k$-th metallic mean, and $\beta=-\frac{1+\alpha}{2\alpha-k}$. These constants $\alpha$
and $\beta$ can be retrieved from Equation~\ref{eq:gaussmap} if we
expand it up to order $O\left(\frac 1n\right)$ as
\[
n\alpha+\beta+O\left(\frac 1 n\right)=kn-1+\frac {n}{\alpha}-\frac{\alpha+\beta}{\alpha^2}+O\left(\frac 1 n\right).
\]
Equating the terms of order $n$ to zero gives that $\alpha$ is the
$j$-th metallic mean. By doing the same for the terms of order $1$,
we get that $\beta=-\frac{\alpha^2+\alpha}{\alpha^2+1}$, which 
gives the expression for $\beta$ in the conjecture.

We extend the one-parameter family in Equation~\ref{eq:Kk} to
the two-parameter family
\begin{equation}\label{eq:zv}
u+v+\mathop{\mbox{\Large\bfseries K}}\limits_{j=1}^{\infty} \frac{j^2}{(j+1)u+v},
\end{equation}
for constants $u$ and $v$.
Its tails satisfy
\[
s_n=nu+v+\frac {n^2}{s_{n+1}}.  
\]
If we assume that $s_n$ converges to $\alpha n+\beta$ with order $O\left(\frac 1n\right)$,
and if we take $u$ to be a natural number,
then $\alpha$ is the $u$-th metallic mean and 
\[
\beta=\frac{v\alpha^2-\alpha}{\alpha^2+1}.
\]
The generalization of Conjecture~\ref{con:bds} then says that the rounded down tails
of the continued fraction give the Beatty sequence with slope $\alpha$ and offset $\beta$.

The continued fraction in Equation~\ref{eq:zv} can be expressed
in terms of hypergeometric functions.
We are indebted to Frits Beukers for pointing
that out to us~\cite{Beukers}. 
The following equality follows from \cite[Eq. 15.7.5]{dlmf}.
\[
\dfrac{c}{{}_2F_1(1,1;c+1;z)}
= c-z
  +\mathop{\mbox{\Large\bfseries K}}\limits_{j=1}^{\infty}
  \dfrac{j^2\,z(1-z)}{c-z+j(1-2z)},
\]
which converges for Re$(z)<\frac 12$,
It is not hard to verify that
\[
\dfrac{\lambda\cdot c}{{}_2F_1(1,1;c+1;z)}
= \lambda(c-z)
  +\mathop{\mbox{\Large\bfseries K}}\limits_{j=1}^{\infty}
  \dfrac{(\lambda j)^2\,z(1-z)}{\lambda(c-z+j(1-2z))}.
\]
If we set
\[
\lambda=\sqrt{4+u^2}\;,\; z=\dfrac 12-\dfrac{u}{2\lambda}\;,\;c=z+\dfrac{u+v}{\lambda}, 
\]
then 
\[
\dfrac{\lambda\cdot c}{{}_2F_1(1,1;c+1;z)}
= u+v+\mathop{\mbox{\Large\bfseries K}}\limits_{j=1}^{\infty} \frac{j^2}{(j+1)u+v}.
\]
For the continued fraction in Equation~\ref{eq:n2n}, the parameters
are $\lambda=\sqrt 5\;,\; z=\frac 12-\frac{1}{2\sqrt 5}\;,\; c=z$.
Beukers' formula in Equation~\ref{eq:n2n} follows from the transformation
\[
{{}_2F_1(a,b;c;z)}=\dfrac{1}{(1-z)^a}\cdot{{}_2F_1(a,c-b;c;z/(z-1))}.
\]
The continued fraction is therefore equal to
$
\tfrac{\sqrt{5}\;(1-\eta)\;\eta}
{{}_2F_1(1,\eta;\eta+1;\eta/(\eta-1))}\;=\;\tfrac{\phi}
{\sum_{n=1}^\infty \tfrac 1{n+\eta}\left(\tfrac {-1}{\phi^2}\right)^n}.
$

\medbreak
We now restrict our attention to one particular case of the conjecture.
The lower Wythoff sequence \seqnum{A000201}, which is the
$(1,1,2,1)$-hiccup sequence, is the Beatty sequence$\lfloor n\phi\rfloor$ for
the golden mean $\phi$, i.e., the first metallic mean. 
According to the generalized conjecture, if we take
$u=1$ and $v=\frac 1\phi=\phi-1$, then the rounded down tails of the
continued fraction are equal to the lower Wythoff sequence.
We prove that the conjecture holds in this particular case.
The backward recurrence for these values of $u$ and $v$ is given by $s_n=n+1/\phi+\frac {n^2}{s_{n+1}},$
which in forward recurrence is equal to
\begin{equation*}
    s_{n+1}=\frac {n^2}{s_{n}-n-1/\phi}.  
\end{equation*}
It is convenient to scale the tails to $r_n=\frac{s_n}n$, which satisfy the recurrence
\begin{equation}\label{eq:gmrecur}
    r_{n+1}=\frac{n}{n+1}\cdot \frac {1}{r_{n}-1-\frac{1}{n\phi}}.  
\end{equation}

\begin{lemma}\label{lem:uniq}
There is a unique sequence of positive real numbers $(r_n)_{n \geq 1}$ that satisfies the recurrence~\ref{eq:gmrecur}.
Moreover, this sequence is contained in $(1+\frac 1{n\phi},2)$.
\end{lemma}
\begin{proof}
    We first prove that a positive sequence $r_n$ must be contained in $(1+\frac{1}{n\phi},2)$. 
    The lower bound $r_n>1+\frac 1{n\phi}$ is required to ensure that $r_{n+1}>0$.
    If $r_n\geq 2$ and $n>1$, then
    \begin{equation}\label{ineq:1}
    r_{n+1}\leq\frac{n}{n+1}\cdot \frac {1}{1-\frac{1}{n\phi}}=\frac{n^2\phi}{n^2\phi+n(\phi-1)-1}<1,
    \end{equation}
    which is impossible since we have already established that $r_{n+1}>1$. If $r_1\geq 2$, then
    $r_2\leq\frac 12\cdot \frac 1{1-\frac 1\phi}=1+\frac{1}{2\phi}$, which is again impossible.

    We denote the functions that produce the recurrence by \[f_n(x)=\frac{n}{n+1}\cdot \frac {1}{x-1-\frac{1}{n\phi}},\] 
    and we denote $I_n=(1+\frac 1{n\phi},2)$. 
    If $n>1$, then by inequality~\ref{ineq:1} the image $f_n(I_n)$ contains $(1,\infty)$. Therefore, $f_n^{-1}(I_{n+1})\subset I_n$ and $f_n^{-1}(I_{n+1})$ is an open subinterval in which the end points are elements of $I_n$.
    We have a descending chain
    \begin{equation}\label{eq:chain}
        I_2\supset f_2^{-1}(I_3)\supset f_2^{-1}(f_3^{-1}(I_4))\supset \cdots        
    \end{equation}
    such that the end-points of each next interval are in the previous interval. 
    Therefore the chain has non-empty intersection. We need to show that this intersection is a singleton $\{r_2\}$.
    
    On the interval $I_n$ the derivate of $f_n$ is bounded by below by its value in $2$: 
    \[
    |f_n'(x)|>\frac {n}{n+1}\cdot\frac{1}{\left(1-\frac 1{n\phi}\right)^2}=\frac{1}{(1+\frac 1n)(1-\frac {2}{n\phi}+\frac 1{n^2\phi^2})}=\frac 1{1-\frac{2-\phi}{n\phi}-\frac {2\phi-1}{n^2\phi^2}+\frac 1{n^3\phi^2}},\] 
    which is greater than $1+\frac{2-\phi}{n\phi}$. Each $I_{n+1}$ has length less than $1$ and the inverse
    $f_n^{-1}$ shrinks its length by a factor that is greater than $1+\frac{2-\phi}{n\phi}$. The product
    of all these factors is infinite and therefore the intersection is indeed a singleton $\{r_2\}$.
    Now $f_1(I_1)\supset I_2$ since $f_1(2)=1+\frac 1{2\phi}$ 
    and there is a unique $r_1\in I_1$ such that $f_1(r_1)=r_2$.
\end{proof}

To prove that $\lfloor s_n\rfloor =\lfloor n\phi\rfloor$ we need to establish that there
is no integer $m$ in between $s_n$ and $n\phi$. This is equivalent to proving that there
is no rational $\frac mn$ in between $r_n$ and $\phi$. To do this, we need to narrow down the
intervals $I_n$ that we found in the previous lemma, which is a straightforward but headache
causing computation. 

\begin{lemma}\label{lem:estimate}
Let $\psi=-\frac {1}{\phi}$.
For each $n\geq 1$ we have $r_n\in \left(\phi+\frac{1}{n(n+1/2)\sqrt 5},\phi+\frac{1-1/{(5n^2)}}{n^2\sqrt 5}\right)$.
\end{lemma}
\begin{proof}
    Let $J_n=\left(\phi+\frac{1}{n(n+1/2)\sqrt 5},\phi+\frac{1-\psi^{4n}}{n^2\sqrt 5}\right)$. We prove that \[f_n(J_n)\supset J_{n+1}.\] Since $f_n$ is decreasing, we need to verify the two inequalities
    \begin{equation}\label{left}
    f_n\left(\phi+\frac{1}{n(n+1/2)\sqrt 5}\right)>\phi+\frac{1-\psi^{4(n+1)}}{(n+1)^2\sqrt 5},        
    \end{equation}
    and
    \begin{equation}\label{right}
    f_n\left(\phi+\frac{1-\psi^{4n}}{n^2\sqrt 5}\right)<\phi+\frac{1}{(n+1)(n+3/2)\sqrt 5}.
    \end{equation}
    We have that 
\begin{align*}
f_n(\phi + x) &= \frac{n}{n+1} \cdot \frac{1}{\phi + x - 1 - \frac{1}{n\phi}} 
= \frac{n}{n+1} \cdot \frac{1}{x + \frac{n-1}{n\phi}} \\
&= \frac{n^2\phi}{(n+1) n \phi x + n^2 - 1} 
= \frac{\phi}{1 + \frac{\phi x (n+1)}{n} - \frac{1}{n^2}}\\
&=\frac{\phi}{1 - \frac{1-\phi x n(n+1)}{n^2}}.
\end{align*}
    For the first inequality~\ref{left} we have $x=\frac{1}{n(n+1/2)\sqrt 5}$ and therefore we need to verify that
    \[
    \frac{\phi}{1-\frac {\sqrt 5(n+1/2) -\phi(n+1)}{n^2(n+1/2)\sqrt 5}}>\phi+\frac{1-1/(5(n+1)^2)}{(n+1)^2\sqrt 5}.
    \]
    We will prove the sharper inequality
    \[
    \frac{\phi}{1-\frac {\sqrt 5(n+1/2) -\phi(n+1)}{n^2(n+1/2)\sqrt 5}}>\phi+\frac{1}{(n+1)^2\sqrt 5}.
    \]
    This can be rewritten as
    \[
    \frac {\phi(\sqrt 5(n+1/2) -\phi(n+1))}{n^2(n+1/2)\sqrt 5}>\frac{1}{(n+1)^2\sqrt 5}\left(1-\frac {\sqrt 5(n+1/2) -\phi(n+1)}{n^2(n+1/2)\sqrt 5}\right).
    \] 
    Since $\phi(\sqrt 5 -\phi)=1$ this is
    \[
    \frac{n+1-\phi/2}{n+1/2}>\frac{n^2}{(n+1)^2}\left(1-\frac {\sqrt 5(n+1/2) -\phi(n+1)}{n^2(n+1/2)\sqrt 5}\right).
    \]
    Both factors on the right-hand side are less than $1$ and the inequality follows from
    \[
    \frac{n+1-\phi/2}{n+1/2}>\frac{n^2}{(n+1)^2},
    \]
    which reduces to $(n+1-\phi/2)(n+1)^2>n^2(n+1/2)$. It is straightforward to check that this holds
    by comparing the coefficients of the polynomials.
    
    For the
    second inequality in~\ref{right} we have that $x=\frac{1-1/(5n^2)}{n^2\sqrt 5}$ and we need to verify that 
    \[
    \frac{\phi}{1-\frac {n\sqrt 5 -\phi(n+1)(1-1/(5n^2))}{n^3\sqrt 5}}<\phi+\frac{1}{(n+1)(n+3/2)\sqrt 5}.
    \]
    For $n=1$ the left-hand side is $1.397...$ and the right-hand side is $1.707...$. For $n=2$ 
    the left-hand side is $1.605...$ and the right-hand side is $1.660...$. 
    %For $n=3$ we get $1.628...$ and $1.642...$.
    We may assume that $n>2$.
    
    The inequality can be rewritten as
    \[
    \frac{\phi(n\sqrt 5 -\phi(n+1)(1-1/(5n^2)))}{n^3\sqrt 5}<\frac{1}{(n+1)(n+3/2)\sqrt 5}\left(1-\frac {n\sqrt 5 -\phi(n+1)(1-1/(5n^2))}{n^3\sqrt 5}\right),
    \]
    which simplifies to
    \[
    \frac{
    n-\phi^2+\phi^2/{(5n)}+\phi^2/(5n^2)}n<\frac{n^2}{(n+1)(n+3/2)}\left(1-\frac {n/\phi -\phi +\phi/(5n)+\phi/(5n^2)}{n^3\sqrt 5}\right).
    \]
    The left-hand side is \[1-\phi^2\cdot \frac{1}{n}+O\left(\frac 1{n^2}\right)\] and the right-hand side is \[1-\frac 52 \cdot \frac 1n+O\left(\frac 1 {n^2}\right)\] and therefore the inequality holds for sufficiently large $n$. We need to verify that it
    holds for all $n$. On the right-hand side of the equation, the factor 
    \[
    1-\frac {n/\phi -\phi +\phi/(5n)+\phi/(5n^2)}{n^3\sqrt 5}
    \]
    is less than $1$ if $n>2$. If we ignore that factor, we get
    \[
    \frac{
    n-\phi^2+\phi^2/{(5n)}+\phi^2/(5n^2)}n<\frac{n^2}{(n+1)(n+3/2)}.
    \]
Since $\frac{n^2}{(n+1)(n+3/2)}>1-\frac 52 \cdot\frac 1n$ it suffices to show that
    \[
    \frac{
    n-\phi^2+\phi^2/{(5n)}+\phi^2/(5n^2)}n=1-\phi^2\cdot\frac 1n +\frac{\phi^2}{5n^2+5n^3}<1-\frac 52 \cdot\frac 1n,
    \]
    We arrive at
    \[
    \frac{\phi^2}{5n+5n^2}<\frac 1\phi-\frac 12 \text{ for }n>2.
    \]
    It suffices to check this at $n=3$, when the left-hand side is $0.04...$ and the right-hand side is $0.11..$.
    The inequality holds.
\end{proof}

The intervals $J_n$ are disjoint and descend on $\phi$, which implies that
$r_n$ is a decreasing sequence with limit point $\phi$. We need to show
that $(\phi,r_n)\subset (\phi,\phi+\frac{1-1/(5n^2)}{\sqrt 5 n^2})$ does not contain a rational of denominator $n$.
\begin{lemma}
    No rational satisfies $0<\frac mn-\phi<\frac{1-1/(5n^2)}{\sqrt 5 n^2}$.
\end{lemma}
\begin{proof}
    The convergents of $\frac 1\phi$ are ratios of consecutive Fibonacci
    numbers $\frac {F_{k}}{F_{k+1}}$. By Cassini's identity
    \[
    F_k^2-F_{k-1}\cdot F_{k+1}=(-1)^k,
    \]
    which means that consecutive convergents are Farey neighbors. 
    The convergents $\frac {F_{2k}}{F_{2k+1}}$ are smaller than $\frac 1\phi$ and the
    convergents $\frac {F_{2k-1}}{F_{2k}}$ are greater.
    By Binet's formula $F_k=\frac{\phi^k-\psi^k}{\sqrt 5}$ with $\psi=-1/\phi$
    \[
    \frac {F_{2k-1}}{F_{2k}}-\frac 1\phi=\frac{(\phi-\psi)\psi^{2k}}{\phi^{2k}-\psi^{2k}}=\frac{\psi^{2k}}{F_{2k}}.
    \]
    Now $\psi^{2k}F_{2k}=\frac{1-\psi^{4k}}{\sqrt 5}$ and therefore
    \[
    \frac {F_{2k-1}}{F_{2k}}-\frac 1\phi=\frac {1-\psi^{4k}}{\sqrt 5 F_{2k}^2}.
    \]
    Suppose that $0<\frac mn-\phi<\frac{1-1/(5n^2)}{\sqrt 5 n^2}$,
    which is equivalent to $0<\frac {m-n}n-\frac 1\phi<\frac{1-1/(5n^2)}{\sqrt 5 n^2}$.
    Let $F_{2k}$
    be the largest even-index Fibonacci number $\leq n$. We will show that $\frac{m-n}n$
    is closer to $\frac 1\phi$ than $\frac{F_{2k-1}}{F_{2k}}$, which is impossible,
    since its first Farey neighbor that is larger than $\frac 1\phi$ has denominator
    $F_{2k+2}$. We will show that
    \[
    \frac{1-1/(5n^2)}{\sqrt 5 n^2}<\frac{1-\psi^{4k}}{\sqrt 5 F_{2k}^2}.
    \]
    Now $\psi^{4k}F_{2k}^2<\frac 15$ and therefore it suffices to show that
    \[
    \frac{1-1/(5n^2)}{\sqrt 5 n^2}\leq\frac{1-1/(5F_{2k}^2)}{\sqrt 5 F_{2k}^2}.
    \]
    This follows from the fact that $f(x)=\frac{1-1/(5x^2)}{x^2}$ is
    descending for $x>1$.
\end{proof}
These lemmas prove our result.
\begin{theorem}
    The sequence of rounded down tails $\lfloor s_n\rfloor$ for the continued fraction
    in Equation~\ref{eq:zv} with $z=1$ and $v=1/\phi$
    is equal to the lower Wythoff sequence. 
\end{theorem}

This settles a particular case of the BDS conjecture. It depends on the computation in Lemma~\ref{lem:estimate},
which shows that the tails of the continued fraction of polynomial type with partial fractions $\frac{n^2}{n+1/\phi}$ converge
to $\phi$ slightly faster than the convergents of its regular continued fraction. A proof of the BDS
conjecture requires a conceptual explanation rather than a computation.

\section{Conclusion and acknowledgements}

We have extended the considerations by Bosma, Dekking, and Steiner~\cite{bosma2018} 
from the $(1,1,3,2)$-hiccup sequence to other hiccup sequences.
There is
still much to be explored, most notably the conjectured relation between
polynomial continued fractions and Beatty sequences.

\medbreak
We would like to thank Frits Beukers, Benoit Cloitre, 
and Slade Sanderson for useful suggestions.
Gandhar Joshi was partly supported by LMS travel grant SC7-2425-18.
We conclude with a table of new hiccup sequences that we encountered in our
study.

\begin{table}[h!]
\centering
\small
\begin{tabular}{|c|c|c|c|c|}
\hline
\textbf{OEIS} & $j$ & $x$ & $y$ & $z$ \\
\hline
\seqnum{A000201} &1&1&2&1\\
\seqnum{A001950} &1&2&2&3\\
\seqnum{A003156} &1& 1 & 3 & 1 \\
\seqnum{A026352}&1&1&2&3\\
\seqnum{A026356}&0&2&2&3\\
\seqnum{A284753}&0&2&4&2\\
\hline
\end{tabular}

\caption{Table of $(j,x,y,z)$-hiccup sequences that are not yet recognized as
such in the OEIS at the time of writing. 
}\label{tbl:end}
\end{table}

\bibliographystyle{siam}
\bibliography{cloitre}

\begin{thebibliography}{10}

\bibitem{adamczewski2004}
{\sc B.~Adamczewski}, {\em Symbolic discrepancy and self-similar dynamics}, Ann. Inst. Fourier, 54 (2004), pp.~2201--2234.

\bibitem{alloucheshallit}
{\sc J.-P. Allouche and J.~Shallit}, {\em Automatic sequences: theory, applications, generalizations}, Cambridge {U}niversity {P}ress, 2003.

\bibitem{allouche2022}
{\sc J.-P. Allouche, J.~Shallit, and R.~Yassawi}, {\em How to prove that a sequence is not automatic}, Expo. Math., 40 (2022), pp.~1--22.

\bibitem{hiccup}
{\sc J.~K. Aronson}, {\em Meyler's side effects of drugs used in anesthesia}, Elsevier, 2008.

\bibitem{bauer2005}
{\sc F.~L. Bauer}, {\em Lamberts {K}ettenbruch}, Informatik-Spektrum, 28 (2005), pp.~303--309.

\bibitem{Berstel}
{\sc J.~Berstel and P.~S{\'e}{\'e}bold}, {\em Sturmian words}, Algebraic combinatorics on words,  (2002), pp.~40--97.

\bibitem{Beukers}
{\sc F.~Beukers}.
\newblock private communication, 2025.

\bibitem{bosma2018}
{\sc W.~Bosma, M.~Dekking, and W.~Steiner}, {\em A remarkable sequence related to $\pi $ and $\sqrt {2}$}, Integers,  (2018), p.~A4.

\bibitem{carton2025}
{\sc O.~Carton, J.-M. Couvreur, M.~Delacourt, and N.~Ollinger}, {\em Linear recurrence sequence automata and the addition of abstract numeration systems}, in International Conference on Combinatorics on Words, vol.~15729 of LNCS, Springer, 2025, pp.~70--82.

\bibitem{cloitre2025}
{\sc B.~Cloitre}, {\em A study of a family of self-referential sequences}, arXiv preprint arXiv:2506.18103,  (2025).

\bibitem{cloitre2003}
{\sc B.~Cloitre, N.~Sloane, and M.~J. Vandermast}, {\em Numerical analogues of {A}ronson's sequence}, J. Integer Seq., 6 (2009), p.~03.2.2.

\bibitem{ruskey2006}
{\sc C.~Deugau and F.~Ruskey}, {\em Complete k-ary trees and generalized meta-fibonacci sequences}, DMTCS, Proc. AG (2006), pp.~203--214.

\bibitem{fogg}
{\sc N.~P. Fogg}, {\em Substitutions in Dynamics, Arithmetics and Combinatorics}, vol.~1794 of Lecture Notes in Mathematics, Springer-Verlag, Berlin, 2002.
\newblock Edited by V. Berth{\'e}, S. Ferenczi, C. Mauduit and A. Siegel.

\bibitem{beatty}
{\sc A.~S. Fraenkel}, {\em The bracket function and complementary sets of integers}, Can. J. Math., 21 (1969), pp.~6--27.

\bibitem{fraenkel1982}
\leavevmode\vrule height 2pt depth -1.6pt width 23pt, {\em How to beat your {W}ythoff games' opponent on three fronts}, Amer. Math. Monthly, 89 (1982), pp.~353--361.

\bibitem{genitrini2020}
{\sc A.~Genitrini, B.~Gittenberger, M.~Kauers, and M.~Wallner}, {\em Asymptotic enumeration of compacted binary trees of bounded right height}, J. Comb. Theory, Ser. A, 172 (2020), p.~105177.

\bibitem{hofstadter2008}
{\sc D.~R. Hofstadter}, {\em Metamagical themas: Questing for the essence of mind and pattern}, Basic books, 2008.

\bibitem{jonesthron}
{\sc W.~B. Jones and W.~J. Thron}, {\em Continued fractions}, Encyclopedia of Mathematics and its Applications, 11 (1980).

\bibitem{mousavi}
{\sc H.~Mousavi}, {\em Automatic theorem proving in {W}alnut},  (2016).
\newblock Online \url{https://arxiv.org/abs/1603.06017}.

\bibitem{dlmf}
{\sc {NIST Digital Library of Mathematical Functions}}, {\em Section 15.7: Continued fractions}.
\newblock \url{https://dlmf.nist.gov/15.7}, 2025.
\newblock Accessed: 2025-09-08.

\bibitem{licofage}
{\sc N.~Ollinger}, {\em Licofage software tool}.
\newblock \url{https://pypi.org/project/licofage/}, 2024.
\newblock Available online.

\bibitem{schaeffer2024}
{\sc L.~Schaeffer, J.~Shallit, and S.~Zorcic}, {\em Beatty sequences for a quadratic irrational: Decidability and applications}, arXiv:2402.08331,  (2024).

\bibitem{shallit2022}
{\sc J.~Shallit}, {\em The logical approach to automatic sequences: {E}xploring combinatorics on words with {W}alnut}, vol.~482, Cambridge University Press, 2022.

\bibitem{metallic}
{\sc V.~d. Spinadel}, {\em The metallic means family and forbidden symmetries}, Int Math J, 2 (2002), pp.~279--88.

\end{thebibliography}

\bigskip
\hrule
\bigskip
\noindent 2020 {\it Mathematics Subject Classification}:
Primary 11B85 ; Secondary 11A55.

\noindent \emph{Keywords}: hiccup sequence, morphic sequence, Beatty sequence, continued fraction. 

\bigskip
\hrule
\bigskip

\end{document}